\documentclass[leqno,12pt]{amsart}
\usepackage{amscd,amsfonts,amsmath,amssymb,amstext,amsthm}
\usepackage[utf8]{inputenc}
\usepackage[T2A]{fontenc}
\usepackage{cite}
\usepackage[unicode]{hyperref}
\usepackage[warn]{mathtext}
\usepackage{color}
\usepackage{xcolor}

\makeatletter
\def\@settitle{\begin{center}%
    \baselineskip14\p@\relax
    \bfseries
    \MakeUppercase{\@title}
  \end{center}%
}
\makeatother
\newtheorem{theorem}{Theorem}
\newtheorem{lemma}{Lemma}

\newtheorem{proposition}{Proposition}
\newtheorem{corollary}{Corollary}

\def\C{{\mathbb C}}
\def\P{{\mathbb P}}
\def\Z{{\mathbb Z}}

\def\R{{\mathbb R}}
\def\codim{{\rm codim}}
\def\conv{{\rm conv}}
\def\vol{{\rm vol}}
\def\re{{\rm Re\:}}

\def\E{\:{\rm e}}
\def\ES{{\rm ES}}

\begin{document}

\title{Hermitian mixed volume and an average number of common zeros of holomorphic functions
}
\author{Boris Kazarnovskii}
\address {Institute for Information Transmission Problems \newline
19 B.Karetny per., 127994, Moscow, Russia,\newline
{\rm B.K.:}
{\it kazbori@gmail.com}.}
\thanks{MSC 32A60: 53C65, 58B20}
\begin{abstract}
Let $V_i$ be a finite dimensional Hermitian vector
space of holomorphic
sections of a line bundle $L_i$ on a complex $n$-dimen\-sional manifold $X$.
We associate to $V_i$ the
non-negative Hermitian quadratic form $g_i$ on $X$,
define a  Hermitian mixed volume of $X$
for a "mixing tuple" of $n$ non-negative Hermitian forms,
and prove that
the average number of common zeroes of $f_1\in V_1,\ldots, f_n\in V_n$
equals to the mixed volume of $X$ for the "mixing tuple"
$g_1,\ldots,g_n$.
This note is related to
\href{https://arXiv.org/abs/1802.02741}{arXiv:1802.02741},
where the average number of common zeros for real equations
are treated in a similar way.
\end{abstract}
\maketitle
\section{Introduction}\label{s1}
Let $L_1,\cdots,L_n$ be holomorphic line bundles on a complex manifold $X$
 of dimension $n$
and let $V_i\subset\Gamma(X,L_i)$ be a finite dimensional subspace of holomorphic sections,
such that
\begin{equation}\label{1}
\forall i\leq n,\:x \in X \ \exists s\in V_i: s(x) \ne 0.
\end{equation}
Assume that each $V_i$ has a fixed
Hermitian inner product. 
Let $\P_i=\P(V_i)$ be a projectivization of $V_i$ equipped
with the corresponding Fubini-Study metric $\Phi_i$,
such that  $\vol(\P_i)=1$.
Let $U$ be a relatively compact domain in $X$.
For $(s_1, \ldots, s_n) \in \P_1 \times \ldots \times \P_n$ we denote by
$N_U(s_1, \ldots, s_n)$ the number of isolated zeros of intersection
of hypersurfaces $s_i=0$ in $U$.
We call
$$
{\mathfrak M}_U(V_1, \ldots, V_n) = \int _{\P_1\times \ldots \times \P_n}N_U(s_1,\ldots,s_n)\,ds_1\cdot\ldots\cdot ds_n
$$
the \emph{average number of common zeros in $U$} of $n$ sections $f_i\in V_i$.

Let $A_i(x)=\{f\in V_i\colon f(x)=0\}$.
From (\ref{1}) it follows that 
$\forall x\in X, i\leq n\colon$ $\codim\: A_i(x)=1$.
So we get the mapping $\theta_i\colon X \to\P_i^*$
assigning to $x\in X$ the subspace $A_i(x)$.
Consider the pull-back $g_i=\theta^*_i(\Phi^*_i)$ of dual to $\Phi_i$
Fubini-Study metric $\Phi^*_i$ on $\P^*_i$.

For any tuple $h_1,\ldots,h_n$ of non-negative Hermitian forms on $X$
we define the mixed Hermitian volume $\vol^H_{h_1,\ldots,h_n}(U)$;
see Section \ref{s2}.
\begin{theorem}\label{thm1}
$
{\mathfrak M}_U(V_1, \ldots, V_n) =n!\:\vol^H_{g_1,\ldots,g_n}(U)
$
\end{theorem}
\noindent
Actually the theorem coincides with the Crofton formula
for the product of projective spaces \cite{Ka1,Ka2}; see Section \ref{s2}.
The sources of the theorem are the well-known BKK formula
\cite{BKK},
and also a recent result on the average number of roots
of systems of real equations \cite{we}.
\section{Crofton formula and Hermitian mixed volume}\label{s2}
\noindent
Let  $\theta_i\colon X \to\P_i^*$ be a mapping
defined in Section \ref{s1}.
Denote by $\omega_i$ the pull-back
of the Fubini-Study form on $\P^*_i=\P(V^*_i)$
under the mapping $\theta_i$.
Then the Crofton formula for the product of projective spaces
tells that
$$
 {\mathfrak M}_U(V_1,\ldots,V_n)=\int_U \omega_1\wedge\ldots\wedge\omega_n.
$$
Below, any non-negative Hermitian quadratic form $h$ on $X$ is called a Hermitian metric.
Recall that $h$ is called non-negative if its eigenvalues are non-negative.
Let $W_+$ be a cone of Hermitian metrics.
Define a function $\vol_U\colon W_+\to\R$
as $\vol_U(h)=\vol_h(U)$.
Recall that $f\colon W_+\to\R$ is called a homogeneous polynomial of degree $k$
if for any $g_1,g_2\in W_+$ the function $f(\lambda_1 g_1+\lambda_2 g_2)$
is a homogeneous of degree $k$ in positive variables $\lambda_1,\lambda_2$.
\begin{lemma}\label{lPol}
$\vol_U$ is a homogeneous polynomial of degree $n$
on $W_+$.
\end{lemma}
\begin{proof}
By Wirtinger's theorem,
$\vol_U(h)=\frac{1}{n!}\int_U\omega^n$,
were $\omega={\rm Im}(h)$.
Let $\omega_i={\rm Im}(h_i)$ for $h_1,\ldots,h_n\in W_+$.
Therefore, the function $\vol_U$ extends to a multilinear symmetric $n$-form
$$
\vol^H_{h_1,\ldots,h_n}(U)=\frac{1}{n!}\int_U\omega_1\wedge\ldots\wedge\omega_n.
$$
in variables $h_1,\ldots,h_n$,
such that $\vol^H_{h,\ldots,h}(U)=\vol_U(H)$.
\end{proof}
\noindent
We call $\vol^H_{h_1,\ldots,h_n}(U)$ the
\emph{Hermitian mixed volume} of $U$ for the "mixing tuple" $h_1,\ldots,h_n$.
Now Theorem \ref{thm1} follows from the Crofton formula.
\section{Example: exponential sums}\label{s3}
\noindent
\emph{Exponential sum} (\ES) is a function
in $\C^n$ of the form
$$f(z)=\sum_{\lambda\in\Lambda\subset\C^{n*},\:c_\lambda\in\C}c_\lambda \E^{\langle z,\lambda\rangle},$$
where $\Lambda$ is a finite subset of the dual space $\C^{n*}$.
The set $\Lambda$ is called the \emph{support} of \ES.
The \emph{Newton polytope} of the support
is a convex hull $\conv(\Lambda)$ of $\Lambda$.

Let $X=\C^n$.
We consider the trivial line bundles $L_1,\ldots,L_n$,
and the spaces $V_i\subset\Gamma(L_i,\C^n)$ of \ES s with the supports $\Lambda_i$.
Choose the Hermitian inner product in $V_i$ as
$$
\langle\sum_{\lambda\in\Lambda_i}a_\lambda\E^{\langle z,\lambda\rangle},\sum_{\lambda\in\Lambda_i}b_\lambda\E^{\langle z,\lambda\rangle}\rangle=
\sum_{\lambda\in\Lambda_i}a_\lambda\bar b_\lambda.
$$
Then, by Theorem \ref{thm1},
$$
{\mathfrak M}_U(V_1, \ldots, V_n) =\frac{n!}{(2\pi)^n}\int_U
dd^c\log\sum_{\lambda\in\Lambda_1}\E^{2\re\langle z,\lambda\rangle}\wedge\ldots\wedge
dd^c\log\sum_{\lambda\in\Lambda_n}\E^{2\re\langle z,\lambda\rangle}.
$$
\begin{lemma}\label{lMA1}
The function $\frac{1}{2t}\sum_{\lambda\in\Lambda}\E^{2\re\langle tz,\lambda\rangle}$,
with $t\to+\infty$,
uniformly converges to the support function $h(z)$ of the Newton polytope of $\Lambda$.
\end{lemma}
\noindent
\emph{Proof.}
Let $\mu(z)\in\Lambda$,
such that $\forall \lambda\in\Lambda\colon\:\re\langle z,\mu(z)\rangle\geq\re\langle z,\lambda\rangle$.
Then
$$
 \frac{1}{2t}\log\sum_{\lambda\in\Lambda}\E^{2\re\langle tz,\lambda\rangle} =
h(z)+\frac{1}{2t}(\#\Lambda+\log\sum_{\lambda\in\Lambda}\E^{2\re\langle tz,\lambda-\mu(z)\rangle})=
h(z)+ o(1).
$$
\noindent
\begin{proposition}[see \cite{Ka3}]
With increasing $t$ the differential form
$$
\frac{1}{t^n}\frac{1}{(2\pi)^n}dd^c\log\sum_{\lambda\in\Lambda_1}\E^{2\re\langle z,\lambda\rangle}\wedge\ldots\wedge
dd^c\log\sum_{\lambda\in\Lambda_n}\E^{2\re\langle z,\lambda\rangle}
$$
converges to the correctly defined positive current
$$\frac{1}{\pi^n} dd^ch_1\wedge\ldots\wedge dd^ch_n,$$
where $h_i$ is a support function of the Newton polytope of $\Lambda_i$.
\end{proposition}
Let $K_1,\ldots,K_n\subset\C^n$ be convex bodies with support
functions $h_1,\ldots,h_n$, and $B\subset\C^n$ be a unit ball centered in $0$.
We call
$$V(K_1,\cdots,K_n)=\int_Bdd^ch_1\wedge\ldots\wedge dd^ch_n$$
by a mixed pseudo-volume of
convex bodies $K_1,\ldots,K_n\subset\C^n$.
For the correctness of definition and for some geometrical properties
of a mixed pseudo-volume see \cite{Alesk,Ka1,Ka2,Ka3}.
\begin{corollary}
$$\lim_{t\to+\infty} \frac{\mathfrak M_{tB}(V_1, \ldots, V_n)}{t^n}= \frac{n!}{\pi^n}V(\gamma_1,\cdots,\gamma_n)$$
where $\gamma_i$ is a Newton polytope of $\Lambda_i$.
\end{corollary}
\begin{proposition}[see \cite{Alesk,Ka1,Ka2}]
If $K_1,\cdots,K_n\subset\re\C^n$
then
$$V(K_1,\cdots,K_n)=\vol(K_1,\cdots,K_n),$$
where $\vol(K_1,\cdots,K_n)$ is a mixed volume of convex bodies $K_i$.
\end{proposition}
For $K_1,\ldots,K_n\subset\Z^n\subset\re\C^n$
the BKK theorem \cite{BKK} follows; see \cite{BKK}.

\end{document}